\documentclass[12pt]{amsart}
\usepackage{amssymb,amscd}
\usepackage{verbatim}

\usepackage{amsmath,amssymb,graphicx,mathrsfs}   
\usepackage[colorlinks=true,allcolors = blue]{hyperref} 

\usepackage{tikz}
\usetikzlibrary{matrix}

\usepackage[all]{xy}

\let\frak\mathfrak

\def\>{\relax\ifmmode\mskip.666667\thinmuskip\relax\else\kern.111111em\fi}
\def\<{\relax\ifmmode\mskip-.333333\thinmuskip\relax\else\kern-.0555556em\fi}
\def\vsk#1>{\vskip#1\baselineskip}
\def\vv#1>{\vadjust{\vsk#1>}\ignorespaces}
\def\vvn#1>{\vadjust{\nobreak\vsk#1>\nobreak}\ignorespaces}

  \let\ssize\scriptstyle
\let\sssize\scriptscriptstyle

\let\Medskip\medskip
\def\medskip{\par\Medskip}
\let\Bigskip\bigskip
\def\bigskip{\par\Bigskip}

\let\Maketitle\maketitle
\def\maketitle{\Maketitle\thispagestyle{empty}\let\maketitle\empty}

\newtheorem{thm}{Theorem}[section]
\newtheorem{cor}[thm]{Corollary}
\newtheorem{lem}[thm]{Lemma}

\newtheorem{exmp}[thm]{Example}

\theoremstyle{definition}                                  
\numberwithin{equation}{section}

\theoremstyle{definition}
\newtheorem*{rem}{Remark}

\let\mc\mathcal
\let\nc\newcommand

\let\al\alpha

\let\dl\delta

\let\la\lambda

\let\phi\varphi
\let\si\sigma

\let\Om\Omega

\let\der\partial

\let\ox\otimes

\let\geq\geqslant

\let\leq\leqslant

\let\on\operatorname
\let\bi\bibitem
\let\bs\boldsymbol

\def\C{{\mathbb C}}
\def\Z{{\mathbb Z}}
\def\R{{\mathbb R}}

\def\F{{\mathbb F}}   

\def\+#1{^{\{#1\}}}

\def\id{{\on{id}}}

\def\beq{\begin{equation}}
\def\eeq{\end{equation}}
\def\be{\begin{equation*}}
\def\ee{\end{equation*}}

\nc{\bea}{\begin{eqnarray*}}
\nc{\eea}{\end{eqnarray*}}
\nc{\bean}{\begin{eqnarray}}
\nc{\eean}{\end{eqnarray}}

\let\ga\gamma
\let\Ga\Gamma

\nc{\Il}{{\mc I_{\bs\la}}}
\nc{\bla}{{\bs\la}}
\nc{\Fla}{\F_\bla}
\nc{\tfl}{{T^*\Fla}}
\nc{\GL}{{GL_n(\C)}}
\nc{\GLC}{{GL_n(\C)\times\C^*}}

\let\sd s 

\def\ddk_#1{\kk_{#1}\<\>\frac\der{\der\<\>\kk_{#1}}}

\def\FFF{\mathbb{F}}

\def\bul{\mathbin{\raise.2ex\hbox{$\sssize\bullet$}}}
\def\intt{\mathchoice
{\mathop{\raise.2ex\rlap{$\,\,\ssize\backslash$}{\intop}}\nolimits}
{\mathop{\raise.3ex\rlap{$\,\sssize\backslash$}{\intop}}\nolimits}
{\mathop{\raise.1ex\rlap{$\sssize\>\backslash$}{\intop}}\nolimits}
{\mathop{\rlap{$\sssize\<\>\backslash$}{\intop}}\nolimits}}

\let\kk q 
\let\cc c

\let\Ko K

\def\GZ/{Gelfand-Zetlin}
\def\KZ/{{\slshape KZ\/}}
\def\qKZ/{{\slshape qKZ\/}}
\def\XXX/{{\slshape XXX\/}}

\nc{\A}{{\mc C}}

\def\Sing{{\on{Sing}}}
\def\sll{{\frak{sl}}}

\def\slt{{\frak{sl}_2}}

\nc{\hsl}{\widehat{{\frak{sl}_2}}}

\nc{\BC}{{ \mathbb C}}
\nc{\lra}{\longrightarrow}
\nc{\CO}{{\mathcal{O}}}
\nc{\BZ}{{ \mathbb Z}}
\nc{\hfn}{\hat{\frak{n}}}

\begin{document}

\hrule width0pt
\vsk->

\title[Determinant  of $\F_p$-hypergeometric solutions]
{Determinant of $\F_p$-hypergeometric solutions
\\
under ample reduction}

\author[Alexander Varchenko]
{ Alexander Varchenko$^{\star}$}

\maketitle

\begin{center}
{\it Department of Mathematics, University
of North Carolina at Chapel Hill\\ Chapel Hill, NC 27599-3250, USA\/}

\vsk.5>
{\it Faculty of Mathematics and Mechanics, Lomonosov Moscow State
University\\ Leninskiye Gory 1, 119991 Moscow GSP-1, Russia\/}

\vsk.5>
 {\it $^{ \star}$ Moscow Center of Fundamental and Applied Mathematics
\\ Leninskiye Gory 1, 119991 Moscow GSP-1, Russia\/}

\end{center}

\vsk>
{\leftskip3pc \rightskip\leftskip \parindent0pt \Small
{\it Key words\/}:  KZ equations,  reduction to
characteristic $p$, $\F_p$-hypergeometric solutions

\vsk.6>
{\it 2010 Mathematics Subject Classification\/}: 13A35 (33C60, 32G20) 
\par}

{\let\thefootnote\relax
\footnotetext{\vsk-.8>\noindent
$^\star\<$
{\it E\>-mail}:
anv@email.unc.edu\,, supported in part by NSF grant DMS-1954266}}

\begin{abstract}
We consider the KZ differential equations over $\C$
in the case, when the hypergeometric solutions are 
one-dimensional integrals. We also consider
the same differential equations over a finite field $\F_p$.
We study  the  polynomial solutions of these differential equations over $\F_p$, 
constructed in a previous work joint with 
V.\,Schechtman  and called  the $\F_p$-hypergeometric solutions.

The dimension of the 
space of  $\F_p$-hypergeometric solutions depends on the prime number $p$.
We say that the KZ equations have ample reduction for a prime $p$, if the
dimension of  the space of  $\F_p$-hypergeometric solutions is maximal possible,
that is, equal to  the dimension
of the space of solutions of the corresponding KZ equations over $\C$. Under the assumption of  ample reduction,
we prove a determinant formula for the matrix of coordinates of basis   $\F_p$-hypergeometric solutions.
The formula is  analogous to the corresponding formula for the determinant of the matrix of
coordinates of basis
complex hypergeometric solutions,
in which binomials $(z_i-z_j)^{M_i+M_j}$
are  replaced with $(z_i-z_j)^{M_i+M_j-p}$ and the Euler gamma function $\Ga(x)$ is replaced with
a suitable $\F_p$-analog $\Ga_{\F_p}(x)$ defined on $\F_p$.

\end{abstract}

{\small\tableofcontents\par}

\setcounter{footnote}{0}
\renewcommand{\thefootnote}{\arabic{footnote}}

\section{Introduction}

The KZ equations were introduced  in  \cite{KZ} as  the differential equations satisfied by 
conformal blocks on sphere in the Wess-Zumino-Witten model of conformal field theory.
The hypergeometric solutions of the KZ equations were constructed more than  30 years ago,
see \cite{SV1, SV2}. 
The KZ equations and the hypergeometric solutions 
are related to many subjects in algebra, representation theory,
theory of integrable systems, enumerative geometry, to name a few.

 The polynomial
solutions of the KZ equations over the finite field $\F_p$  
with a prime number $p$ of elements were constructed  recently 
 in \cite{SV3}, see also \cite{V4, V5, V6, V7}. 
   We call these  solutions  the {\it $\F_p$-hypergeometric solutions}.
The general problem is to understand relations between the hypergeometric solutions 
of the KZ equations over $\C$ and the $\F_p$-hypergeometric solutions and observe how the
remarkable properties of hypergeometric solutions are reflected in the properties
of the $\F_p$-hypergeometric solutions. This program is in the first stages, where
we consider  essential  examples and study the corresponding 
 $\F_p$-hypergeometric solutions by direct methods.

\vsk.2>
In this paper we consider the KZ differential equations 
in the case, when the hypergeometric solutions over $\C$ are 
one-dimensional integrals.

\vsk.2>

The dimension of the 
space of  $\F_p$-hypergeometric solutions depends on the prime number $p$.
We say that the KZ equations have ample reduction for a prime $p$, if the
dimension of  the space of  $\F_p$-hypergeometric solutions is maximal possible,
that is, equal to the dimension
of the space of solutions
 of the corresponding KZ equations over $\C$. Under the assumption of  ample reduction,
we prove a determinant formula for the matrix of coordinates of basis   $\F_p$-hypergeometric solutions.
The formula is  analogous to the corresponding formula for the determinant of the matrix of
coordinates of basis
complex hypergeometric solutions, see \cite{V1},
in which binomials $(z_i-z_j)^{M_i+M_j}$
are  replaced with $(z_i-z_j)^{M_i+M_j-p}$ and the Euler gamma function $\Ga(x)$ is replaced with
a suitable $\F_p$-analog $\Ga_{\F_p}(x)$ defined on $\F_p$.

\vsk.4>

In Section \ref{sec KZ} we describe our KZ equations and their reduction modulo $p$.  We define the hypergeometric solutions over $\C$ and $\F_p$-hypergeometric solutions.  
The ample reduction is defined in  Section \ref{sec am red}. 

As  mentioned earlier, the  $\F_p$-hypergeometric solutions are polynomials. In Section \ref{sec coeff} we give a formula
for their coefficients. 

In Section \ref{sec beta}, we consider the particular case of our KZ  equations, 
whose space of solutions over $\C$ is one-dimensional, with the basis solution 
given by the Euler beta integral. We describe the corresponding $\F_p$-hypergeometric solution,
 which we call the {\it $\F_p$-beta integral}.

In Section \ref{sec lea} we consider an arbitrary polynomial solution (not necessarily  $\F_p$-hyper\-geometric) 
of our KZ 
 equations over $\F_p$ and describe 
its leading term with respect to the lexicographical ordering of monomials, see Theorem \ref{lem C0}.
It turns out that the notion of leading term and the formula for the leading term
in Theorem \ref{lem C0} 
 are useful in studying polynomial 
solutions of the KZ  equations over $\F_p$. The notion of  leading term replaces, to some extend,
 the notion of  initial condition, when  the differential equations are over $\F_p$.

The module of $\F_p$-hypergeometric solutions has a natural basis. In Section \ref{sec le co} we describe the leading 
terms of the basis  $\F_p$-hypergeometric solutions.  Section \ref{sec det} contains our main result,
Theorem \ref{thm det}, describing the determinant of coordinates of the basis $\F_p$-hypergeometric solutions,
under assumption of ample reduction.

In Section \ref{ex1} we give an example of KZ  equations and a prime $p$,
such that the space of complex solutions is one-dimensional, the space of polynomial solutions over $\F_p$ is 
one-dimensional,
and the KZ equations  have no $\F_p$-hypergeometric solutions.
In Section \ref{sec last} we show that if the reduction of our KZ  equations is ample for a prime $p$, then all polynomial solutions are $\F_p$-hypergeometric.
\vsk.2>
The author thanks Alexey Slinkin for useful discussions.

\section{KZ equations}
\label{sec KZ}

\subsection{Description of equations} 
\label{sec DE} 

In this paper the numbers $p$, $q$ {\it are prime numbers,
$n$ is a positive integer, $p>n\geq 2$, \,$p>q$.}
We fix {\it a vector $(m_1,\dots, m_n) \in \Z^n_{>0}$, such that 
$m_i<q$ for all $i=1,\dots,n$, }
and study the system of equations
for a  column vector  $I(z)=(I_1(z)$, \dots, $I_{n}(z))$\,:
\bean
\label{KZ}
\phantom{aaa}
 \frac{\partial I}{\partial z_i} \ = \
   {\frac 1q} \sum_{j \ne i}
   \frac{\Omega_{ij}}{z_i - z_j}  I ,
\quad i = 1, \dots , n,
\qquad
m_1I_1(z)+\dots+m_nI_{n}(z)=0,
\eean
where $z=(z_1,\dots,z_n)$,
the $n\times n$-matrices $\Om_{ij}$ have the form:
\bean
\label{Om_ij_reduced}
 \Omega_{ij} \ = \ \begin{pmatrix}
             & \vdots^i &  & \vdots^j &  \\
        {\scriptstyle i} \cdots & {-m_j} & \cdots &
            m_j   & \cdots \\
                   & \vdots &  & \vdots &   \\
        {\scriptstyle j} \cdots & m_i & \cdots & {-m_i}&
                 \cdots \\
                   & \vdots &  & \vdots &
                   \end{pmatrix} ,
\eean                    
and all other entries are zero.
 This  joint system of {\it differential and 
algebraic equations} is called the {\it system of KZ  equations} in this paper.

\vsk.2>
\begin{rem}

System of equations \eqref{KZ} is the system of  standard KZ differential equations with parameter $q$, 
associated with the Lie algebra $\sll_2$ and the subspace of singular 
vectors of weight $\sum_{i=1}^nm_i-2$ of the tensor 
product
$V_{m_1}\ox\dots\ox V_{m_n}$, 
where $V_{m_i}$ is the irreducible $m_i+1$ dimensional  $\sll_2$-module, up to a gauge transformation, see 
this example in  \cite[Section 1.1]{V3}.

\end{rem}

\vsk.2>
We consider system \eqref{KZ} over the field $\C$ and over the field $\F_p$ with $p$ elements.

\subsection{Solutions  over $\C$} 
Consider the {\it master function}
\bean
\label{mast f}
\Phi(t,z_1,\dots,z_{n}) = \prod_{a=1}^{n}(t-z_a)^{-m_a/q}
\eean
and  the column ${n}$-vector  of hypergeometric  integrals
\bean
\label{Iga}
I^{(\ga)} (z)=(I_1(z),\dots,I_n(z)),
\eean
 where
\bean
\label{s}
I_j=\int  \Phi(t,z_1,\dots,z_{n})\, \frac {dt}{t-z_j},\qquad j=1,\dots,{n}\,.
\eean
The integrals $I_j$, $j=1,\dots,n$, are over an element $\ga$ of the first homology group
 of the algebraic curve with affine equation
\bea
y^q = (t-z_1)^{m_1}\dots (t-z_{n})^{m_n}\,.
\eea
Starting from such $\ga$,  chosen for given values
$\{z_1,\dots,z_{n}\}$, the vector $I^{(\ga)}(z)$ can 
be analytically continued as a multivalued holomorphic function of $z$ to the complement in $\C^n$ of the union of the
diagonal  hyperplanes $z_i=z_j$, $i\ne j$.

\vsk.2>
The complex vector space of such integral solutions is the  $n-1$-dimensional vector space of all solutions of
system \eqref{KZ}. 
See these statements in the example in   \cite[Section 1.1]{V3}, also in  \cite{SliV1}, see also the determinant formula
\eqref{det C} below.

\subsection{$\F_p$-Integrals}

 Let $P(x_1,\dots,x_k)$ be a polynomial
with coefficients in an $\F_p$-module,
\bea
P(x_1,\dots,x_k) = \sum_{d_1,\dots,d_k}\, c_{d_1,\dots,d_k}\, x_1^{d_1}\dots x_k^{d_k}.
\eea
Let $(l_1,\dots,l_k)\in \Z_{>0}^k$. The coefficient
$c_{l_1p-1,\dots,l_kp-1}$ is called the {\it $\F_p$-integral over cycle $[l_1,\dots,l_k]_p$}
and denoted by 
\bea
\int_{[l_1,\dots,l_k]_p} P(x_1,\dots,x_k)\,dx_1\dots dx_k.
\eea
We have an analog of Stokes' Theorem:
\bea
\int_{[l_1,\dots,l_k]_p} \frac{\der P}{\der x_i}(x_1,\dots,x_k)\,dx_1\dots dx_k \,=\, 0\,
\eea
for any $[l_1,\dots,l_k]_p$.

\subsection{Solutions  over $\F_p$}

Polynomial solutions of system \eqref{KZ}, considered over the field $\F_p$, 
were constructed in \cite{SV3}.

For $i=1,\dots,n$, choose  the least positive integers $M_i$ such that 
\bean
\label{M_i}
M_i\equiv -\frac{m_i}{q}
\qquad 
(\on{mod} \,p)\,.
\eean
Let
\bean
\label{mp red}
\Phi_{p}(t,z, M) 
&:=&
 \prod_{i=1}^n(t-z_i)^{M_i},
\\
\label{P}
P(x,z) 
&:=&
\Big(\frac {\Phi_p(x,z)}{x-z_1}, \dots,\frac {\Phi_p(x,z)}{x-z_n}\Big)\,=\, \sum_i P^{i}(z) \,x^i \,,
\eean
where $P(x,z)$ is considered as a column $n$-vector of polynomials in $x,z_1,\dots,z_n$ and
$P^{i}(z)$ as column $n$-vectors of polynomials in $z_1,\dots,z_n$ with coefficients in $\FFF_p$.

\vsk.2>

For a positive integer $l$, denote
\bean
\label{md}
I^{[l]}(z)\,:=\, \int_{[l]_p}\Big(\frac {\Phi_p(x,z)}{x-z_1}, \dots,\frac {\Phi_p(x,z)}{x-z_n}\Big)\,dx.
\eean

\begin{thm}[{\cite[Theorem 1.2]{SV3}}] \label{thm Fp} 
For any positive integer $l$, the vector of polynomials $I^{[l]}(z)$
is a solution of KZ system \eqref{KZ}.

\end{thm}

The solutions  $I^{[l]}(z)$
given by this construction are called the {\it $\F_p$-hypergeometric solutions} of equations \eqref{KZ}.

\begin{rem}
The polynomial $\Phi_p(t,z)$ is an $\F_p$-analog of the master function $\Phi(t,z)$.
The polynomial $P(t,z)$ is  an analog of the integrand of  integral \eqref{Iga}.
The transformation  $P(t,z)\to I^{[l]}(z)$ is  an analog of the integral
and the index $[l]_p$ is an analog of the integration cycle.

\end{rem}

\vsk.2>
Denote $\F_p[z^p]:=\F_p[z_1^p,\dots,z_{n}^p]$. 
The set of all polynomial solutions of system \eqref{KZ} with coefficients in $\F_p$
is a module over the ring
$\F_p[z^p]$ since equations \eqref{KZ} are linear and
$\frac{\der z_i^p}{\der z_j} =0$ in $\F_p[z]$ for all $i,j$.
\vsk.2>

The $\F_p[z^p]$-module
\bean
\label{Def M_M}
\mathcal{M}\,=\,\Big\{ \sum_{l} c_l(z) I^{[l]}(z) \ |\ c_l(z)\in\F_p[z^p]\Big\},
\eean
spanned by $\F_p$-hypergeometric solutions,
 is called the {\it module of $\F_p$-hypergeometric solutions}.
\vsk.2>

The range for the index $l$ is defined by  the inequalities 
 $0< lp-1\leq \sum_{i=1}^nM_i-1$. Hence
$l=1,\dots, r$, where
\bean
\label{def dL}
r \,:=\,\Big[ \sum_{i=1}^n M_i / p \Big],
\eean
the integer part of the number $ \sum_{i=1}^n M_i / p$.

\begin{thm}
\label{thm li in}
The $\F_p$-hypergeometric solutions  $I^{[l]}(z)$, $l=1,\dots,r$, are linearly independent over $\F_p[z^p]$.
\end{thm}

\begin{proof} The proof coincides with the proof of \cite[Theorem 3.1]{V5},  see also the proof of 
 {\cite[Theorem 3.2]{SliV1}}. Other proofs see in \cite[Section 4.1]{V6}  and in Section \ref{sec cor} below.
\end{proof}

\begin{lem}
\label{lem am}
We have $r<n$.
\qed

\end{lem}

\subsection{Ample reduction} 
\label{sec am red}

 We say that system \eqref{KZ} has {\it ample reduction} for a prime $p$ if
\bean
\label{def am}
\Big[ \sum_{i=1}^n M_i / p \Big] = n-1\,,
\eean
that is, the rank $r=\Big[ \sum_{i=1}^n M_i / p \Big]$ of the module 
$\mc M$ of $\F_p$-hypergeometric solutions takes the possible maximum
value $n-1$.

\begin{exmp} 

If $q>n$, $p=lq+q-1$ for some $l \in\Z_{>0}$ and $m_i=1$, $i=1,\dots, n$. Then 
$M_i = ((q-1)p-1)/q=p-l-1$ and
\bea
\Big[ \sum_{i=1}^n \frac{M_i}p \Big] =
\Big[ n \frac{(q-1)p-1}{pq} \Big] =
\Big[ n - \frac nq-\frac n{pq}\Big] = n-1 .
\eea
Hence under these assumptions system \eqref{KZ} has ample reduction.

\end{exmp}

\begin{lem}
\label{lem am ine}
If system \eqref{KZ} has ample reduction for a prime $p$, then  for any $l=1,\dots, n-1$
and any subset $I\subset\{1,\dots,n\}$ with $|I|=l$,  we have
\bean
\label{am l ine}
(l-1) p \,<\,\sum_{i\in I} M_i\, <\, lp\,.
\eean

\end{lem}

\begin{proof} 
The second inequality holds since $0<M_i<p$ for any $i$.  Assume that the first inequality is not true and
 $\sum_{i\in I} M_i\leq (l-1)p$ for some $l$ and $I$. Then 
\bea
\sum_{i=1}^n M_i \leq (l-1)p + \sum_{i\in \bar I} M_i < (l-1)p + (n-l)p = (n-1)p ,
\eea
where $\bar I$ is the complement of $I$. That contradicts to the ampleness of the reduction.
\end{proof}

\subsection{Example}
\label{ex1}

 Let $n=2$, $p=3$, $q=2$, $m_1=m_2=1$. Then
$M_1=M_2=1$. The KZ equations take the form
\bea
\frac{\der I}{\der z_1} = \frac{\Om}{z_1-z_2}I, 
\qquad
\frac{\der I}{\der z_2} = \frac{\Om }{z_2-z_1}I, 
\qquad
\Om=\begin{pmatrix}
\,\,\,1 & -1  
\\
-1&\,\,\,\,1     
\end{pmatrix}.
\eea
The polynomial 
$(z_1-z_2)^2\!
\begin{pmatrix}
\,\,\,\,1 
\\
-1     
\end{pmatrix}$ is a solution. At the same time $r=[(M_1+M_2)/p]= 0$ and  there are no $\F_p$-hypergeometric solutions.

\section{Coefficients of polynomials}
\label{sec coeff}

\subsection{Coefficients of $\F_p$-hypergeometric solutions}

For $l=1,\dots,r$, the coordinates of the column vector 
$I^{[l]}(z) =(I^{[l]}_1(z), \dots, I^{[l]}_1(z))$
are homogeneous polynomials in $z_1,\dots,z_n$ of degree 
\bean
\label{deg P}
\delta_{l}:=
\sum_{j=1}^n M_j - lp\,.
\eean
Let
\bea
I^{[l]}(z) = \sum_{d_1,\dots,d_n} I^{[l]}_{d_1,\dots,d_n} z_1^{d_1}\dots z_n^{d_n}\,,
\qquad I^{[l]}_{d_1,\dots,d_n} \in \F_p^n\,.
\eea

\begin{lem}
\label{lem coef} 
The coefficient $I^{[l]}_{d_1,\dots,d_n}\in \F_p^n$ is nonzero if and only if
\bean
\label{cond ne 0}
\sum_{i=1}^n d_i = \delta_{l}, \quad\on{and}\quad d_i\leq M_i \quad\on{for}\quad i=1,\dots,n\,,
\eean
moreover,
\bean
\label{Coe}
I^{[l]}_{d_1,\dots,d_n} \,=\,(-1)^{\dl_l}
\prod_{j=1}^n\binom{M_j}{d_j}\,
\Big(1 - \frac{d_1}{M_1}, \dots , 1 - \frac{d_n}{M_n}\Big).
\eean
If  $(I^{[l]}_{d_1,\dots,d_n; 1}, \dots,  I^{[l]}_{d_1,\dots,d_n;n})$ are coordinates of
$I^{[l]}_{d_1,\dots,d_n}$,   then
\bean
\label{sing M}
\sum_{i=1}^n \,m_i\,  I^{[l]}_{d_1,\dots,d_n; i} \,=\,\sum_{i=1}^n \,M_i\,  
I^{[l]}_{d_1,\dots,d_n; i} \,=\,0\,.
\eean

\end{lem}

\begin{proof} The first statements follow from formulas \eqref{mp red} and \eqref{P}. 
Formula \eqref{sing M} follows from formulas \eqref{Coe}, \eqref{M_i}.
\end{proof}

\subsection{Coefficients and singular vectors}
Consider the Lie algebra $\slt$ over the field $\F_p$ with standard generators $e,f,h$
and relations
$[e,f]=h$, $[h,e]=2e$, $[h,f]=-2f$.

For $m\in \Z_{\geq 0}$, $m<p$, let $V_m$ be the irreducible $\slt$-module over $\F_p$ with highest weight $m$,
basis $f^jv_m$, $j=0,\dots,m$,  and standard $\slt$-action. 

Consider the $\slt$-module $\ox_{j=1}^nV_{m_j}$.
For $i=1,\dots,n$, let
\bean
\label{f^i}
f^{(i)}v := v_{m_1}\ox\dots\ox v_{m_{i-1}}\ox fv_{m_{i}}\ox  v_{m_{i+1}}\ox\dots\ox  v_{m_{n}}\
\in \ \ox_{j=1}^nV_{m_j}\,.
\eean
Then
\bea
h f^{(i)}v = \Big(\sum_{j=1}^nm_j - 2\Big) f^{(i)}v\,,
\qquad 
e f^{(i)}v =m_i\,v_{m_1}\ox\dots \ox v_{m_i}\ox\dots\ox  v_{m_{n}}\,.
\eea
Denote  by $V[-2]$ the $n$-dimensional subspace of $\ox_{j=1}^nV_{m_j}$ generated by 
$f^{(i)}v$, $i=1,\dots,n$. Denote
\bea
\Sing V[-2] = \Big\{ \sum_{i=1}^n c_if^{(i)}v\ \Big|\ \sum_{i=1}^n c_im_i =0\Big\}\,\subset\, V[-2]\,.
\eea
The $n-1$-dimensional subspace 
 $\Sing V[-2] \subset V[-2]$ is the kernel of the restriction to $V[-2]$ of the operator
$e : \ox_{j=1}^nV_{m_j} \to \ox_{j=1}^nV_{m_j}$\,.
Denote 
\bea
\Sing V[-2][z]\ := \ \Sing V[-2]\ox_{\F_p} \F_p[z]\,.
\eea
Define an isomorphism of vector spaces
\bean
\label{iota}
\iota : \F_p^n \to V[-2],\qquad (c_1,\dots,c_n) \mapsto  \sum_{i=1}^n c_if^{(i)}v \,.
\eean
Then an $\F_p$-hypergeometric solution $I^{[l]}(z)$ is identified with the polynomial
\bea
\iota I^{[l]}(z) := \sum_{d_1,\dots,d_n} \iota I^{[l]}_{d_1,\dots,d_n}\, z_1^{d_1}\dots z_n^{d_n}\,
\in\ \Sing V[-2][z]\,.
\eea

\subsection{Operators $\Om_{ij}^{\frak{sl}_2}$}

The isomorphism $\iota$ identifies a  linear operator $\Om_{ij} : \F_p^n \to \F_p^n$, appearing
in  system \eqref{KZ},
with a linear operator on $V[-2]$, which we denote by $\iota \Om_{ij}$. Namely, the linear operator $\iota \Om_{ij}$
is the restriction to $V[-2]$ of the Casimir operator
on  $\ox_{j=1}^nV_{m_j}$ defined by the formula
\bea
\Om_{ij}^\slt\,: = \,\frac 12 h^{(i)}h^{(j)} + e^{(i)}f^{(j)} +f^{(i)}e^{(j)}\, - \frac{m_im_j}2\on{Id}\,,
\eea
where for $x\in \slt$ we define the operator $x^{(i)}$ on  $\ox_{j=1}^nV_{m_j}$  by
\bea
x^{(i)} := 1\ox\dots \ox 1\ox x\ox 1\ox\dots\ox 1\,
\eea 
with $x$ at the $i$th position. Notice that each $\Om_{ij}^\slt$ preserves $\Sing V[-2]$.

\section{$\F_p$-Beta integral and KZ equations for $n=2$}
\label{sec beta}

\subsection{Solutions over $\C$} 
\label{sec bC}

Consider  the system of KZ equations \eqref{KZ} over  $\C$ for $n=2$. Then the master function is
\bean
\label{mast f2}
\Phi(t,z_1,z_2) = (t-z_1)^{-m_1/q}(t-z_2)^{-m_2/q}\,, 
\eean
and the  one-dimensional space of solutions is generated by the 2-column vector
\bean
\label{I2}
I(z_1,z_2) = \int_{z_1}^{z_2}\Phi(t,z) \Big(\frac {f^{(1)}v}{t-z_1} + \frac {f^{(2)}v}{t-z_2}\Big) dt.
\eean
To determine this integral we assume that $z_1, z_2$ are real, $z_1<z_2$, and 
fix a univalued branch on $[z_1,z_2]$ of each of the factors $(t-z_1)^{-m_1/q}$, $(t-z_2)^{-m_2/q}$.
Then 
\bean
\label{bC}
&&
\\
\notag
&&
\!\!\!\!\!\!\!
I(z_1,z_2) = (z_2-z_1)^{-m_1/q}(z_1-z_2)^{-m_2/q}
\frac{\Ga(-m_1/q+1) \Ga(-m_1/q+1)}{\Ga(-m_1/q-m_2/q+1)} 
\Big(\frac {f^{(1)}v}{-m_1/q} -  \frac {f^{(2)}v}{-m_2/q}\Big),
\eean
\vsk.2>
\noindent
where $(z_k-z_l)^{-m_l/q}$ is the value of the chosen branch of the  function $(t-z_l)^{-m_l/q}$ at $t=z_k$.

\begin{rem}
 Calculation of each 
coordinate of the vector $I(z_1,z_2)$ is reduced to the beta integral, 
\bean
\label{be int}
\int_{0}^{1} t^{\al-1}(1-t)^{\beta-1} dt =\frac{\Ga(\al)\Ga(\beta)}{\Ga(\al+\beta)}\ ,
\eean
\vsk.2>
\noindent
after change of variables. Formula \eqref{bC} shows how 
 the  beta integral  appears in hypergeometric solutions
 of KZ equations.


\end{rem}

\vsk.4>

If $(-m_1/q,-m_2/q) = (M_1,M_2)$, where $M_1,M_2$ are positive integers, then
\bean
\label{bC2}
I(z_1,z_2) = (z_2-z_1)^{M_1}(z_1-z_2)^{M_2}
\frac{\Ga(M_1+1)\,\Ga(M_2+1)}{\Ga(M_1+M_2+1)} 
\Big(\frac {f^{(1)}v}{M_1} - \frac {f^{(2)}v}{M_2}\Big).
\eean

\subsection{Factorial and gamma functions}
Recall that $p$ is an odd prime number. The $p$-adic factorial function is
defined on positive integers by
\bea
(x!)_p \,=\, \prod_{1 \leq j \leq x, \ ( j, p) = 1}\,j\ .
\eea
The Morita $p$-adic gamma function is the unique continuous function of a $p$-adic integer $x$ (with values in 
$\Z_p$)  such that
\bea
\Ga_p(x)  = (-1)^x \prod_{1 \leq j < x, \  ( j, p) = 1}\  j \,,
\eea
for positive integers $x$. Thus $\Ga_p(x+1) = (-1)^x (x!)_p$ for positive integers $x$.

Define the function
\bea
\Ga_{\F_p} : \Z_{\geq0} \to \F_p
\eea
by setting $\Ga_{\F_p}(0)=1$, $\Ga_{\F_p}(1)=-1$ and 
mapping  an integer $x>1$ to the  image of the integer $\Ga_{p}(x)$ in $\F_p$.

\begin{lem}

We have  $\Ga_{\F_p}(x+p)= \Ga_{\F_p}(x)$ for all $x$.

\end{lem}

\begin{proof} The lemma follows from Wilson's theorem, $(p-1)! \equiv -1$ (mod $p$).
\end{proof}

We extend the function $\Ga_{\F_p}$ to the set $\Z$ by periodicity,   $\Ga_{\F_p}(x+p)= \Ga_{\F_p}(x)$.
Then we get
\bea
\Ga_{\F_p}(x)\,\Ga_{\F_p}(1-x) \,= \,(-1)^x
\eea
also by Wilson's theorem.

\begin{lem}
\label{lem ga}
Let $A,B$ be positive integers such that $A<p$, $B<p$, $p\leq A+B$. Then we have an identity in $\F_p$,
\bean
\label{p-ga}
  B\,\binom{B-1}{A+B-p} 
  &=&
B\,\binom{B-1}{p-A-1}
\\
\notag
& =& 
 (-1)^{A+1}\,\frac{A! \,B!}
{(A+B-p)!}
=(-1)^{A}\,\frac{\Ga_{\F_p}(A+1)\Ga_{\F_p}(B+1)}
{\Ga_{\F_p}(A+B-p+1)}    \,.
\eean
\qed

\end{lem}

\begin{proof}
We have
\bea
&&
B\binom{B-1}{p-A-1} = \frac{B(B-1) \cdots(A+B-p+1)}
{1\cdots (p-A-1)}
\\
&&
= \frac{B \cdots (A+B-p+1) (A+B-p)! A!}
{(-1)^{p-A-1} (p-1)(p-2) \cdots  (A+1)\,A!(A+B-p)!}
\\
&&
= (-1)^{A+1}\,\frac{A! \,B!}
{(A+B-p)!} = (-1)^{A} \,\frac{\Ga_{\F_p}(A+1)\,\Ga_{\F_p}(B+1)}
{\Ga_{\F_p}(A+B-p+1)}\,.
\eea
\end{proof}

\subsection{$\F_p$-hypergeometric solutions}
\label{sec soF}

Consider  the system of KZ equations \eqref{KZ} over  $\F_p$ for $n=2$.
Assume that system \eqref{KZ} has ample reduction for a prime $p$. Then the integers $M_1$, $M_2$,
introduced in \eqref{M_i}, satisfy the inequalities
\bean
\label{Mn2}
0\,< M_1, \,M_2,\,  M_1+M_2-p +1\, <\, p\,.
\eean
In this case the module 
$\mc M$ of $\F_p$-hypergeometric solutions is of rank one and generated by
$I^{[1]}(z_1,z_2)$. The solution $I^{[1]}(z_1,z_2)$ is the coefficient of $x^{p-1}$ in the Taylor expansion
of the polynomial 2-vector
\bean
\label{Fp2}
P(x,z_1,z_2) =
(x-z_1)^{M_1}(x-z_2)^{M_2} \Big(\frac {f^{(1)}v}{x-z_1} + \frac {f^{(2)}v}{x-z_2}\Big),
\eean
cf. \eqref{I2}

\begin{thm}
\label{thm beta}
We have
\bean
\label{p beta} 
\phantom{aaa}
I^{[1]}(z_1,z_2)
&=& (-1)^{M_2}
 (z_2-z_1)^{M_1+M_2-p}
\frac{\Ga_{\F_p}(M_1+1)\,\Ga_{\F_p}(M_2+1)}{\Ga_{\F_p}(M_1+M_2-p+1)}
\Big(\frac {f^{(1)}v}{M_1} - \frac {f^{(2)}v}{M_2}\Big)
\\
\notag
&=& (-1)^{M_1}
 (z_1-z_2)^{M_1+M_2-p}
\frac{\Ga_{\F_p}(M_1+1)\,\Ga_{\F_p}(M_2+1)}{\Ga_{\F_p}(M_1+M_2-p+1)}
\Big(\frac {f^{(2)}v}{M_2}-\frac {f^{(1)}v}{M_1}\Big),
\eean
cf. \eqref{bC2}.

\end{thm}

\begin{proof} 
Make the transformation
\bea
P(x,z_1,z_2)\
\mapsto \
P(x+z_1,z_1,z_2) =
x^{M_1}(x+z_1-z_2)^{M_2} \Big(\frac {f^{(1)}v}{x} + \frac {f^{(2)}v}{x+z_1-z_2}\Big).
\eea
This change of variables does not change the coefficient of $x^{p-1}$ in the Taylor expansion
 by Lucas theorem, see 
\cite{Lu} and the proof of
\cite[Lemma 5.2]{V5}. Hence
\bea
I^{[1]}(z_1,z_2)\,
=
\,(z_1-z_2)^{M_1+M_2-p}\Big(
\binom{M_2}{p-M_1} f^{(1)}v
+
\binom{M_2-1}{p-M_1-1}f^{(2)}v\Big).
\eea
Then
\bean
\label{bin}
\phantom{aaa}
\binom{M_2}{p-M_1} f^{(1)}v
+
\binom{M_2-1}{p-M_1-1}f^{(2)}v
=
M_2\binom{M_2-1}{p-M_1-1}\Big(\frac {f^{(1)}v}{p-M_1} 
+
\frac{f^{(2)}v}{M_2}\Big).
\eean
Now the theorem follows from Lemma \ref{lem ga}.
\end{proof}

\begin{rem}

For positive integers $a,b$ satisfying the inequalities\
$a<p$, $b<p$,  $p-1\leq a+b$, we have  the {\it$\F_p$-beta integral formula}
\bean
\label{fpbe}
\int_{[1]_p} x^{a}(1-x)^{b} dx \,=\,  \,- \,\frac{a!\,b!}{(a+b-p+1)!}\,,
\eean
which  follows from Lemma \ref{lem ga}.

Formula \eqref{be int} for the beta integral is the one-dimensional case of the
$n$-dimensional Selberg integral formula, see \cite{Se}. In \cite{RV1, RV2} we develop
$\F_p$-analogs of the $n$-dimensional Selberg integral formulas.

\end{rem}

\section{Leading term of a polynomial solution}
\label{sec lea}

\subsection{Lexicographical ordering}

For a permutation $\si=(\si_1,\dots,\si_n)\in S_n$ denote by $>_\si$  the {\it lexicographical ordering} 
of monomials
\bea
z_1^{d_1}\dots z_n^{d_n}\,,\qquad d_1,\dots,d_n\in\Z_{\geq 0}\,,
\eea
relative to the ordering $(\si_1,\dots,\si_n)$ of the integers $(1,\dots,n)$.  So $z_{\si_1} >_\si
z_{\si_2} >_\si\dots$ $>_\si z_{\si_{n-1}} >_\si z_{\si_n}$ and so on.

For a nonzero polynomial
\bea
f(z)=\sum_{d_1,\dots,d_n} a_{d_1,\dots,d_n} z_1^{d_1}\dots z_n^{d_n}\,
\eea
let $f_\si(z)$ be the summand
$a_{d_1,\dots,d_n} z_1^{d_1}\dots z_n^{d_n}$ corresponding to the largest
monomial entering   $f(z)$ with a nonzero coefficient. We call $f_\si(z)$ 
 the {\it $\si$-leading term} of $f(z)$, the corresponding 
$a_{d_1,\dots,d_n}$ the {\it $\si$-leading coefficient}, the corresponding
$z_1^{d_1}\dots z_n^{d_n}$ the {\it $\si$-leading monomial}.

\vsk.2>
In particular, consider elements $f(z)=(f_1(z),\dots,f_n(z)) \in (\F_p[z])^n$ as polynomials
in $z$ with coefficients in $\F_p^n$. Then for any  $\si\in S_n$ a nonzero element $f(z)$ has 
 $\si$-leading term
$f_\si(z) = a_{d_1,\dots,d_n} z_1^{d_1}\dots z_n^{d_n}$ and $\si$-leading coefficient
$a_{d_1,\dots,d_n}\in \F_p^n$.

\subsection{Leading term of a polynomial solution}

Consider  the lexicographical ordering $>_\id$ of monomials 
corresponding to the identity permutation  $\id\in S_n$.
Hence $z_1>_\id z_2>_\id\dots>_\id z_n$ and so on.

\begin{lem}
\label{lem le co}

Let $ I(z)=(I_1(z),\dots, I_n(z))$ be a polynomial solution of system \eqref{KZ} over $\F_p$
(not necessarily  an $\F_p$-hypergeometric solution). Let
$C z_1^{d_1}\dots z_n^{d_n}$, $C=(C_1$,
\dots, $C_n) \in \F_p^n$,
 be the $\id$-leading term of $I(z)$.
  Then
\bean
\label{eqn le}
\phantom{aaa}
\sum_{j=1}^n m_jC_j =0, \quad 
\sum_{l=j+1}^n \,\Om_{jl}\,C \,=\, q d_j\, C\,,\quad j=1, \dots, n-1\,, \quad d_n\equiv 0 \ (\on{mod}\,p)\,,
\eean
where $C$ is considered as a column vector.

\end{lem}

\begin{proof}

Rewrite the KZ equations as
\bean
\label{mKZ}
\Big(\prod_{k\ne j}(z_j-z_k)\Big) \frac{\der I}{\der z_j} 
- \frac 1q \sum_{l\ne j} \Big( \prod_{k\notin \{l,j\}}(z_j-z_k)\Big) \Om_{jl} I\,=\,0\,,
\eean 
$j=1,\dots,n$. Now the lemma follows from calculating  the leading term of the left-hand side in \eqref{mKZ}
 and equating it to zero.
\end{proof}

For $j\ne l$ introduce the $n\times n$-matrices
\bean
\label{Om_ij_M}
 \Omega_{jl}^M \ = \ \begin{pmatrix}
             & \vdots^j &  & \vdots^l &  \\
        {\scriptstyle j} \cdots & {M_l} & \cdots &
            -M_l   & \cdots \\
                   & \vdots &  & \vdots &   \\
        {\scriptstyle l} \cdots & -M_j & \cdots & {M_j}&
                 \cdots \\
                   & \vdots &  & \vdots &
                   \end{pmatrix} 
\eean                    
with all other entries equal to zero. 
For $j=1,\dots,n-1$, denote $\Om_j^M=\sum_{l=j+1}^n \,\Om_{jl}^M$\,.

\begin{cor}
\label{cor le co}
Let $ I(z)$ be a polynomial solution of system \eqref{KZ} over $\F_p$.
Then the  $\id$-leading term $C z_1^{d_1}\dots z_n^{d_n}$ of  $I(z)$ satisfies the system of equations:
\bean
\label{eqn le}
\sum_{j=1}^n M_jC_j =0, \quad 
\Om_{j}^M C \,=\, d_j\, C\,,\quad j=1, \dots, n-1\,, \quad d_n\equiv 0 \ (\on{mod}\,p)\,.
\eean
\qed

\end{cor}

\begin{thm}
\label{lem C0}
Let  a pair $C= (C_1,\dots,C_n),\, (d_1,\dots,d_n)\in \F_p^n$
be a  solution of system
\bean
\label{eqn leq}
\sum_{j=1}^n M_jC_j =0, \quad 
\Om_{j}^M C \,=\, d_j\, C\,,\quad j=1, \dots, n-1\,, \quad d_n=0\,.
\eean
 Let the index $i$ be such that 
\bean
\label{ass C}
C_i
&\ne&
 0 \quad\on{and}\quad C_j=0,
\quad j=1,\dots,i-1.
\eean
 Then
\bean
\label{l pre}
d_j
&=&
M_j\,,\qquad  j=1,\dots, i-1\,,
\\
\notag
d_i  
&=&  \sum_{j=i}^n M_j\,,\qquad d_i\ne M_i\,,
\\
\notag
d_j
&=&
0\,,\qquad j=i+1,\dots, n\,,
\\
\label{3.8}
\phantom{aaa}
\sum_{l=i+1}^n\!M_l
&\ne& 0, \qquad
C_j \,= \,-\, \frac{M_i}{\sum_{l=i+1}^n\!M_l}\, C_i\,,  \qquad j=i+1,\dots,n\,.
\eean
Conversely, if a pair $C= (C_1,\dots,C_n),\, (d_1,\dots,d_n)\in \F_p^n$
has properties  \eqref{ass C},  \eqref{l pre}, \eqref{3.8},
then it is a solution  of system \eqref{eqn leq}.

\end{thm}

\noindent
\begin{proof}
For any $j=1,\dots,n-1$, we have

\bean
\label{OC}
&&
\Om_j^M (C_1,\dots,C_n)=(0,\dots,0,  
\\
\notag
&&
\phantom{aaa}
M_{j+1}(C_j-C_{j+1}) +\dots+M_n(C_j-C_n), M_j(C_{j+1}-C_j), \dots,
M_j(C_n-C_j))\,,
\eean

\vsk.2>
\noindent
where in the right-hand side\ 0 is repeated $j-1$-times,

Assume \eqref{eqn leq} and \eqref{ass C}. First we check that 
$d_j = M_j$, $j=1,\dots, i-1$.  Indeed,
\bean
\label{OC_j}
\Om^M_jC 
&=&
\Big(0,\,\dots, \,0,\, -\sum_{l=j+1}^nM_l C_l,\, M_{j}C_{j+1}\,,\,\dots,\,  M_{j}C_{n}\Big)\,.
\\
&=&
\Big(0,\,\dots, \,0,\, \phantom{aaaa} 0 \phantom{aaaa}, \, M_{j}C_{j+1}\,,\,\dots,\,  M_{j}C_{n}\Big)\, = M_j\,C\,.
\eean
Here we used that $\sum_{l=j+1}^nM_l C_l = \sum_{l=1}^nM_l C_l = 0$.
Hence if  $\Om^M_jC = d_jC$, the  $d_j=M_j$.

We also have
\bean
\label{OC_i}
\phantom{aa}
\Om^M_iC 
&=&
(0,\dots, 0, \,\sum_{l=i+1}^nM_l(C_i- C_l),\, M_{i}(C_{i+1}-C_i),\dots,  M_{j}(C_{n}-C_i))
\\
\notag
&=&
 (0,\dots, 0,\, \phantom{aaa}
 C_i\sum_{l=i}^nM_l\phantom{aaa}  , 
 M_{i}(C_{i+1}-C_i),\,\dots,  M_{j}(C_{n}-C_i))=d_i C\,.
\eean
Hence
\bean
\label{3.12}
d_i =\sum_{l=i}^nM_l\,,\qquad 
 M_{i}(C_{j}-C_i) = C_j\sum_{l=i}^nM_l \,,
\qquad j=i+1,\dots,n\,.
\eean
The second equality  in \eqref{3.12} implies $-M_{i}C_i = C_j\sum_{l=i+1}^nM_l$.
Hence $\sum_{l=i+1}^nM_l\ne 0$. This inequality and the first equality in \eqref{3.12} imply 
$d_i\ne M_i$. Also the second equality in \eqref{3.12} implies that  $C_{i+1}=\dots=C_n$ and,
therefore,
$\Om^M_j C=0$ for $j>i$. Hence $d_j=0$ for $j=i+1,\dots, n-1$.
Thus we deduced \eqref{l pre} and \eqref{3.8}.

The proof that \eqref{ass C}, \eqref{l pre}, \eqref{3.8} imply \eqref{eqn leq} is straightforward.
\end{proof}

\begin{cor}
\label{cor deg le}
Let  a pair $C= (C_1,\dots,C_n),\, (d_1,\dots,d_n)\in \F_p^n$
be a solution of system \eqref{eqn le} such that $C\ne 0$.
 Then
$\sum_{j=1}^n d_j = \sum_{j=1}^n M_j$ \,.
\qed

\end{cor}

\begin{cor}
\label{cor deg sol}
Let $ I(z)=(I_1(z),\dots, I_n(z))$ be a homogeneous polynomial solution of system \eqref{KZ} over $\F_p$
of degree  $d$, then
\bean
\label{deg pol sol}
d \,\equiv\, \sum_{j=1}^n M_j \ \ (\on{mod}\,p),
\eean
cf. Example \ref{ex1}.
\qed

\end{cor}

\begin{cor}
\label{cor number}

Let  a pair $C= (C_1,\dots,C_n),\, (d_1,\dots,d_n)\in \F_p^n$
be a solution of system \eqref{eqn le} such that $C\ne 0$.
Then $(d_1,\dots,d_n)$ is uniquely determined by $C$.
If $c\in \F_p^\times$, then the pair $cC, (d_1,\dots,d_n)$
is also a solution of system \eqref{eqn le}. The number of
equivalence classes of solutions $C, (d_1,\dots,d_n)$
 of system \eqref{eqn le}, modulo the equivalence relation $C\mapsto cC$,
 equals $n-1$ decreased by the number of indices $i$,
 $1\leq i\leq n-2$,  such that
$\sum_{j=i+1}M_j =0$.
\qed

\end{cor}

\subsection{Example} Let
 $n=2$, $p=3$, $q=2$, $m_1=m_2=1$, $M_1=M_2=1$. 
 Then system \eqref{KZ} has
a (non-$\F_p$-hypergeometric) polynomial solution 
 $I(z) =(z_1-z_2)^2\begin{pmatrix}
\,\,\,\,1 
\\
-1     
\end{pmatrix}$, see Example \ref{ex1}.  The leading $\id$-term of $I(z)$ equals 
 $z_1^2\begin{pmatrix}
\,\,\,\,1 
\\
-1     
\end{pmatrix}$ in agreement with Theorem \ref{lem C0}.

 In this example  the module of all polynomial solutions is one-dimensional and generated by 
 $I(z)$.   This follows from the fact that  the $\id$-leading term of any polynomial solution
 $J(z)$  has the form
$cz_1^{a_1}z_2^{a_2}\cdot\,z_1^2\! \begin{pmatrix}
\,\,\,\,1 
\\
-1     
\end{pmatrix}$, where $c\in\F_p^\times$, $a_1,a_2\in p\Z$, by Theorem \ref{lem C0}.

\section{Leading term of an $\F_p$-hypergeometric solution }
\label{sec le co}

\subsection{Leading term of  $I^{[l]}(z)$}

Using the isomorphism $\iota$, defined in \eqref{iota}, we consider \linebreak
$\F_p$-hypergeometric solutions 
 as polynomials in $z$ with coefficients in $\Sing V[-2]$.

\vsk.2>

Recall that for $l=1,\dots,r$, the $\F_p$-hypergeometric solution $I^{[l]}(z)$ is a homogeneous polynomial in $z$
of degree  $\delta_{l} = \sum_{j=1}^n M_j - lp$.

We  describe the leading term $I^{[l]}_\id (z)$ of $I^{[l]}(z)$ with respect to  
  the lexicographical ordering $>_\id$.

\begin{thm}
\label{thm tle}

Given $l=1,\dots,r$, denote by  $i=i(l)$   the unique positive integer such that
\bean
\label{i(l)}
0\,\leq\, \sum_{j=i}^n M_{j}\,-\,lp\, < M_{i}\,.
\eean
Then
\bean
\label{le term}
I^{[l]}_\id(z)
&=&
\,(-1)^{\sum_{j=1}^iM_j}\,\frac{\Ga_{\F_p}(M_i+1)\,\Ga_{\F_p}\!\left(\sum_{j=i+1}^nM_j-(l-1)p+1\right)}
{\Ga_{\F_p}\!\left(\sum_{j=i}^nM_j-lp+1\right)}
\\
\notag
&\times&
\left( \frac {\sum_{j= i+1}^n f^{(j)}v}{\sum_{j= i+1}^n M_j}-\frac{f^{(i)}v}{M_{i}}\right) 
z_{1}^{M_{1}}z_{2}^{M_{2}}\dots z_{{i-1}}^{M_{{i-1}}}
z_{i}^{\sum_{j=i}^n M_{j}\,-\,lp}\,,
\eean
where $f^{(k)}v$ are introduced in \eqref{f^i}.

\end{thm}

\begin{proof}
The theorem follows from formula \eqref{Coe}, Corollary \ref{cor le co}, Theorem  \ref{lem C0} and 
Lemma \ref{lem ga},
 where Lemma \ref{lem ga} is applied with $A=M_i$ and $B= \sum_{j=i+1}^nM_j-(l-1)p$.
\end{proof}

\begin{rem}
Let $\si=(\si_1,\dots,\si_n)\in S_n$. The leading term $I^{[l]}_\si (z)$ of $I^{[l]}(z)$ with respect to  
  the lexicographical ordering $>_\si$ is obtained from the right-hand side of formula \eqref{le term}
by  simultaneous reordering
of variables $z_1,\dots,z_n$ and parameters $M_1,\dots,M_n$.
\end{rem}

\subsection{Indices $i(l)$}

In Theorem \ref{thm tle} we define the numbers $i(l)$, $l=1,\dots,r$. The leading term 
$I^{[l]}_\id(z)$ of the $\F_p$-hypergeometric solution $I^{[l]}(z)$
is a monomial in $z_1,\dots, z_{i(l)}$\,.

\begin{lem}
\label{lem length}
We have
\bean
\label{lengths}
1\leq i(r) < i(r-1) < \dots < i(1)\, < n\,.
\eean
\qed
\end{lem}

\begin{cor}
\label{cor am i}
If system \eqref{KZ} has ample reduction for a prime $p$, then
\bean
\label{am ij}
i(l) = n - l, \qquad l=1,\dots, n-1.
\eean
\qed

\end{cor}

\subsection{Corollary of Theorem \ref{thm tle}} 
\label{sec cor}

The $\F_p$-hypergeometric solutions $I^{[l]}(z)$,
$l=1,\dots,r$, are linearly independent over the field $\F_p(z)$ of rational functions in $z$ with coefficients in $\F_p$.

\vsk.2>
This follows from the fact that the leading coefficients
of $I^{[l]}(z)$,
$l=1,\dots,r$,  are linear independent over the field $\F_p$.

\subsection{Example} Let $q=3$, $p=13$, $(m_1,\dots,m_6)= (2,2,2,1,1,1)$. Then
$(M_1,\dots,M_6)= (8,8,8,4,4,4)$,
$r = \big[ \sum_j M_j/p\big] = 2$. We have two $\F_p$-hypergeometric solutions $I^{[l]}(z)$, $l=1,2$,
of degrees 23 and 10, respectively.

Consider 
 the identity element $\id=(1,2,3,4,5,6)$ of $S_6$  and the elements 
 $s_{3,4}=(1,2,4$, $3,5,6)$,   $\si=(6,5,4,3,2,1)\,\in\,S_6$.
  Then the leading terms are
\bea
I^{[1]}_{\id} (z) 
&=& 
-
\binom{8}{7}
\Big(\big(1 - \frac{7}{8}\big) f^{(3)}v +f^{(4)}v+f^{(5)}v+f^{(6)}v\Big) 
z_1^8z_2^8z_3^7\,, 
\\
I^{[1]}_{s_{3,4}} (z) 
&=& 
-
\binom{4}{3}
\Big(\big(1 - \frac{3}{4}\big) f^{(3)}v +f^{(5)}v+f^{(6)}v\Big) 
z_1^8z_2^8z_4^4z_3^3\,, 
\\
I^{[1]}_{\si} (z) 
&=& 
-
\binom{8}{3}
\Big(\big(1 - \frac{3}{8}\big) f^{(2)}v + f^{(1)}v\Big) 
z_6^4z_5^4z_4^4z_3^8z_2^3\,, 
\eea

\bea
I^{[2]}_{\id} (z) 
&=& 
\binom{8}{2}
\Big(\big(1 - \frac{2}{8}\big) f^{(2)}v +f^{(3)}v +f^{(4)}v+f^{(5)}v+f^{(6)}v\Big) 
z_1^8z_2^2\,, 
\\
I^{[2]}_{{s_{3,4}}} (z) 
&=& 
\binom{8}{2}
\Big(\big(1 - \frac{2}{8}\big) f^{(2)}v +f^{(3)}v +f^{(4)}v+f^{(5)}v+f^{(6)}v\Big) 
z_1^8z_2^2\,, 
\\
I^{[2]}_{\si} (z) 
&=& 
\binom{4}{2}
\Big(\big(1 - \frac{2}{4}\big) f^{(4)}v + f^{(3)}v+f^{(2)}v+f^{(1)}v \Big) 
z_6^4z_5^4z_4^2\,. 
\eea

\section{Determinant of $\F_p$-hypergeometric solutions}
\label{sec det}

\subsection{Determinant over $\C$}

Consider  the system of KZ equations \eqref{KZ} over  $\C$. Then the space of solutions is $n-1$-dimensional.

\vsk.2>
Recall the master function $\Phi(t,z)$ introduced in \eqref{Iga}.
Consider the hypergeometric integrals
\bea
I^{(j)}(z) = \int_{z_{j}}^{z_{j+1}} \!\Phi(t,z) \,\sum_{j=1}^n \frac{f^{(j)}v}{t-z_j}\, dt,  \qquad j=1,\dots, n-1.
\eea
To determine the integrals we assume that $z_1,\dots, z_n$ are real, $z_1<\dots <z_n$, and 
for every $l=1,\dots,n$, we fix  a univalued branch  of the function $(t-z_l)^{-m_l/q}$ on the rays
$\{t\in\R\ |\ t<z_l\}$ and  $\{t\in\R\ |\ t>z_l\}$.

The integrals $I^{(j)}$, $j=1,\dots,n-1$, form a basis of solutions of system \eqref{KZ} due to the following theorem.

\begin{thm} [\cite{V1, V2, V3}]
\label{thm det C}

We have
\bean
\label{det C}
&&
{\det}_{j,l=1}^{n-1} \Big( \frac{-m_{l+1}}q\int_{z_{j}}^{z_{j+1}} \Phi(t,z) \frac{dt}{t-z_{l+1}}\Big)
\\
\notag
&&
\phantom{aaaaaaaaa}
=\ \frac{\Ga(-m_1/q+1) \cdots \Ga(-m_n/q+1)}
{\Ga(-m_1/q-\dots - m_n/q+1)}\prod_{1\leq j,\,l\leq n, \ j\ne l} (z_j-z_l)^{-m_l/q}\,.
\eean

\end{thm}

Below we  prove an analog of this theorem for $\F_p$-hypergeometric solutions.

\subsection{Determinant over $\F_p$}

Introduce a basis of  $\Sing V[-2]$,
\bean
\label{basis}
w_j = \frac{f^{(j)}v}{M_j}- \frac {f^{(j+1)}v}{M_{j+1}} , \qquad j=  1,\dots,n-1\,.
\eean
Expand each of the $\F_p$-hypergeometric solutions $I^{[l]}(z)$
with respect to this basis,
\bean
\label{exp ell}
I^{[l]}(z) = \sum_{j=1}^{n-1} c^l_j(z) w_j,
\eean
where $c^l_j(z)$ are scalar homogeneous polynomials of degree $\sum_{j=1}^nM_j-lp$.

\begin{thm}
\label{thm det}

Assume that system \eqref{KZ} has ample reduction for a prime $p$. Then we have the square matrix
$c(z) = (c^l_j(z))_{l,j=1,\dots,n-1}$ with nonzero determinant,
 given by the following formula
\bean
\label{det}
\phantom{aaa}
\det c(z) 
\,=
\,
\frac{\Ga_{\F_p}(M_1+1)\cdots \Ga_{\F_p}(M_n+1)}
{\Ga_{\F_p}(M_1+\dots +M_n-(n-1)p+1)}
\prod_{1\leq i<j \leq n}(-1)^{M_j}(z_j-z_i)^{M_i+M_j-p}.
 \eean
\end{thm}

Notice that 

\bean
\label{Ga!}
\phantom{aaa}
\frac{\Ga_{\F_p}(M_1+1)\cdots \Ga_{\F_p}(M_n+1)}{\Ga_{\F_p}(M_1+\dots +M_n-(n-1)p+1)}
\,=\, (-1)^{n-1}
\frac{M_1!\cdots M_n!}{(M_1+\dots+M_n-(n-1)p)!}\,.
\eean

\vsk.4>

The theorem is proved in   Sections \ref{5.2}, \ref{5.3}.

\subsection{Preliminary remarks}
\label{5.2}

\begin{lem}
\label{lem nonz}
The function $\det c(z)$ is a nonzero homogeneous polynomial of degree
\bean
\label{deg c}
 (n-1)\sum_{j=1}^n M_j - \frac{n(n-1)}2\,p
\eean
with  $\id$-leading term 
\bean
\label{ld c}
{\det}_\id c(z) = \on{const} z_1^{(n-1)(M_1-p) + \sum_{j=2}^nM_j}  
z_2^{(n-2)(M_2-p) + \sum_{j=3}^nM_j}\dots z_{n-1}^{M_{n-1}-p + M_n}\,,
\eean
where 
\bean
\label{const}
\on{const} =  (-1)^{n(n-1)/2+\sum_{j=1}^{n-1}(n-j)M_j}
\,
\frac{\Ga_{\F_p}(M_1+1)\cdots \Ga_{\F_p}(M_n+1)}{\Ga_{\F_p}(M_1+\dots +M_n-(n-1)p+1)}\,,
\eean

\end{lem}

\begin{proof}
If $\det c(z)$ is nonzero, then it is of degree
\bea
\deg (\det c(z)) = \sum_{l=1}^{n-1}\Big(\sum_{j=1}^nM_j-lp\Big),
\eea
which gives \eqref{deg c}.  By Theorem \ref{thm tle} we have
\bean
\label{lead}
I^{[l]}_\id(z)
&=&
\,(-1)^{\sum_{j=1}^{n-l}M_j}\,\frac{\Ga_{\F_p}(M_{n-l}+1)\,\Ga_{\F_p}\!\left(\sum_{j={n-l}+1}^nM_j-(l-1)p+1\right)}
{\Ga_{\F_p}\!\left(\sum_{j={n-l}}^nM_j-lp+1\right)}
\\
\notag
&\times&
\left( \frac {\sum_{j= {n-l}+1}^n f^{(j)}v}{\sum_{j= {n-l}+1}^n M_j}-\frac{f^{({n-l})}v}{M_{{n-l}}}\right) 
z_{1}^{M_{1}}z_{2}^{M_{2}}\dots z_{{{n-l}-1}}^{M_{{{n-l}-1}}}
z_{{n-l}}^{\sum_{j={n-l}}^n M_{j}\,-\,lp}\,,
\eean
 Expanding vectors $ \Big(\frac {\sum_{j= {n-l}+1}^n f^{(j)}v}
 {\sum_{j= {n-l}+1}^n M_j}-\frac{f^{({n-l})}v}{M_{{n-l}}}\Big)$,
$l=1,\dots,n-1$, with respect to the basis 
$w_1,\dots,w_{n-1}$, we obtain a square $(n-1)\times(n-1)$-matrix $C(\id)=(C(\id)^l_j)$ of coefficients
of the expansion.
The matrix is triangular with respect to the main anti-diagonal. The anti-diagonal entries are
\bea
C(\id)^l_{n-l}  = (-1)^{\sum_{j=1}^{n-l}M_j +1}\,
\frac{\Ga_{\F_p}(M_{n-l}+1)\,\Ga_{\F_p}\!\left(\sum_{j={n-l}+1}^nM_j-(l-1)p+1\right)}
{\Ga_{\F_p}\!\left(\sum_{j={n-l}}^nM_j-lp+1\right)}\,.
\eea
Hence  
\bea
\det C(\id) 
=
 (-1)^{n(n-1)/2 + \sum_{j=1}^{n-1}(n-j)M_j   }
\,
\frac{\Ga_{\F_p}(M_1+1)\cdots \Ga_{\F_p}(M_n+1)}{\Ga_{\F_p}(M_1+\dots +M_n-(n-1)p+1)}\,.
\eea
These formulas imply the lemma.
\end{proof}

\subsection{Differential equations for $\det c(z)$}
\label{5.3}

\begin{lem}
\label{lem deq}

The polynomial $\det c(z)$ satisfies the system of scalar differential equations
\bean
\label{deq det}
\frac{\der y}{\der z_i} = \sum_{j\ne i} \frac{M_i+M_j}{z_i-z_j}\,y, \qquad i=1,\dots, n\,.
\eean

\end{lem} 

\begin{proof}  The operators $\Om_{ij}^M$ defined in \eqref{Om_ij_M} preserve  $\Sing V[-2]$ and 
\bean
\label{Tr ij}
\on{Tr}\big\vert_{\Sing[-2]}\Om_{ij}^M = M_i+M_j .
\eean
The polynomial $\det c(z)$, as the determinant of solutions, satisfies the system of equations
\bean
\label{deq det}
\frac{\der \det c(z)}{\der z_i} = \sum_{j\ne i} \frac{\on{Tr}\big\vert_{\Sing[-2]} \Om_{ij}^M}{z_i-z_j}\,\det
c(z), \qquad i=1,\dots, n\,.
\eean
This proves the lemma.
\end{proof}

\begin{lem}
\label{lem yo}
The polynomial 
\bea
y^0(z) = \prod_{1\leq i<j \leq n}(z_i-z_j)^{M_i+M_j-p}
\eea
is a solution of system \eqref{deq det}.
\qed

\end{lem}

\begin{lem}
\label{lem div}
For any $i\ne j$, the polynomial $\det c(z)$ is divisible by $(z_i-z_j)^{M_i+M_j-p}$.
\end{lem}

\begin{proof}
We will prove that $\det c(z)$ is divisible by $(z_{n-1}-z_n)^{M_{n-1}+M_{n}-p}$. The divisibility by
$(z_i-z_j)^{M_i+M_j-p}$ for other $i\ne j$ is proved by  reordering  variables.

Introduce new variables $u_1,\dots,u_n$ by the equaions: $u_n=z_1+\dots+z_n$,
\bean
\label{zu}
z_{j+1}-z_j = u_1u_2\dots u_j,\qquad j=1,\dots, n-1.
\eean
The variables $z_1,\dots,z_n$ are polynomials in $u_1,\dots,u_n$.
Hence if $I(z)$ is a polynomial solution of system \eqref{KZ}, then $I(z(u))$ is a polynomial
solution of the transformed differential KZ equations, which take the form:
\bean
\label{uKZ}
\frac{\der I}{\der u_i} 
=
\Big( \frac{\sum_{j>i} \Om_{ij}^M}{u_i} + \on{Reg}_i(u)\Big) I, \quad i=1,\dots,n-1,
\qquad
\frac{\der I}{\der u_n} = 0,
\eean
where $\on{Reg}_i(u)$ is an operator-valued function of $u$ regular at the point
$u_1=\dots=u_n=0$. See this statement in \cite[Proposition 2.2.3]{V7}. In particular, we have
\bean
\label{uKZ n-1}
\frac{\der I}{\der u_{n-1}} = \Big(\frac{\Om_{n-1,n}^M}{u_{n-1}} + \on{Reg}_{n-1}\Big) I.
\eean
Hence the polynomial $\det c(z(u))$ a solution of the system of differential equations;
\bean
\label{uKZ}
\frac{\der y}{\der u_i} 
=
 \Big(\frac{\sum_{j>i} (M_i+M_j)}{u_i} + \on{Tr}\on{Reg}_i(u)\Big) y, \quad i=1,\dots,n-1,
\qquad
\frac{\der y}{\der u_n} = 0,
\eean
cf. formulas \eqref{Tr ij}, \eqref{deq det}. In particular, we have
\bean
\label{uKZ n-1}
\frac{\der I}{\der u_{n-1}} = \Big(\frac{M_{n-1}+M_n}{u_{n-1}} + \on{Reg}_{n-1}\Big) I.
\eean
Hence the polynomial $\det c(z(u))$ is divisible by a monomial $u_{n-1}^d$, where $d$ is a
nonnegative solution of the congruence
\bea
d\equiv M_{n-1}+M_n \ \ (\on{mod}\ p) .
\eea
The number $M_{n-1}+M_n-p$ is the smallest nonnegative solution. 
Hence $\det c(z(u))$ is divisible by $u_{n-1}^{M_{n-1}+M_n-p}$, that is,
\bea
\det c(z(u)) &=& F(u_1,\dots,u_n) \,u_{n-1}^{M_{n-1}+M_n-p},
\eea
where $F(u_1,\dots,u_n)$ is a polynomial. Thus
\bea
\det c(z)
&=&
F\Big(z_2-z_1,\frac{z_3-z_2}{z_2-z_1}, \dots, \frac{z_n-z_{n-1}}{z_{n-1}-z_{n-2}}, z_1+\dots+z_n\Big)
\Big(\frac{z_n-z_{n-1}}{z_{n-1}-z_{n-2}}\Big)^{M_{n-1}+M_n-p}.
\eea
Hence $\det c(z)$ is divisible $(z_n-z_{n-1})^{M_{n-1}+M_n-p}$.
The lemma is proved.
\end{proof}

\smallskip
\noindent
{\it Proof of Theorem \ref{thm det}.}  The theorem follows from Lemmas 
\ref{lem div}  and  \ref{lem nonz}.
\qed

\subsection{Remark on the initial value problem}
\label{sec in co}

\begin{cor}
\label{cor triv}
Assume that system \eqref{KZ} has ample reduction for a prime $p$.
For  $x=(x_1,\dots,x_n)\in\F_p^n$ with  distinct coordinates and
$w\in \Sing V[-2]$,  there exist a
unique vector
$(c_1,\dots, c_{n-1})\in\F_p^{n-1}$ such that
\bean
w=\sum_{l=1}^{n-1} c_l I^{[l]}(x)\,.
\eean

\end{cor}

Denote
$
(\F_p^n) ^o =\{x\in\F_p^n\ |\ x\,\on{has\,distinct\,coordinates}\}$.
We have an isomorphism of the two trivial bundles
\bea
\Sing V[-2] \times (\F_p^n) ^o \to (\F_p^n) ^o
\quad  \on{and}\quad
\F_p^{n-1} \times (\F_p^n) ^o \to (\F_p^n) ^o,
\eea
which sends $(w,x)$ to $((c_1,\dots,c_{n-1}),x)$.

\section{Properties of $\F_p$-hypergeometric solutions}
\label{sec last}

In this section we add more properties of $\F_p$-hypergeometric solutions.

\subsection{Uniqueness property}

Given $l=1,\dots, r$, the $\F_p$-hypergeometric solution $I^{[l]}(z)$
 has degree $\sum_{j=1}^n M_j-lp$ and $\id$-leading term
\bean
\label{le ter}
&&
\,(-1)^{M_{i(l)} +1}\frac{\Ga_{\F_p}(M_{i(l)}+1)\,\Ga_{\F_p}\!\left(\sum_{j={i(l)}+1}^nM_j-(l-1)p+1\right)}
{\Ga_{\F_p}\!\left(\sum_{j={i(l)}}^nM_j-lp+1\right)}
\\
\notag
&&
\phantom{aaaaaa}
\times
\left(\frac{f^{({i(l)})}v}{M_{{i(l)}}} - \frac {\sum_{j= {i(l)}+1}^n f^{(j)}v}{\sum_{j= {i(l)}+1}^n M_j}\right) 
z_{1}^{M_{1}}z_{2}^{M_{2}}\dots z_{{{i(l)}-1}}^{M_{{{i(l)}-1}}}
z_{{i(l)}}^{\sum_{j={i(l)}}^n M_{j}\,-\,lp}\,,
\eean
where the number $i(l)$  is defined in Theorem \ref{thm tle}.

\begin{thm}
\label{thm uniq}
If $I(z)$ is a homogeneous polynomial solution of system \eqref{KZ} with $\id$-leading term 
\eqref{le ter}, then $I(z) = I^{[l]}(z)$.

\end{thm}

\begin{proof}  

 By Theorem \ref{lem C0} the $\id$-leading term of the difference $I(z)-I^{[l]}(z)$
has the form
\bean
\label{le di}
(0, \dots, 0, C_k,\dots,C_n) z_1^{a_1}\dots z_{k-1}^{a_{k-1}}z_k^{a_k}\,,
\eean
for some $k$, where $C_k\ne 0$  and 
\bean
\label{cle di}
&&
a_j\equiv M_j\ (\on{mod} \ p), \quad j=1,\dots,k-1,
\qquad
a_k\equiv \sum_{j=k}^n M_j \ (\on{mod}\ p),
\qquad a_k\not\equiv M_k\,,
\\
\label{cled}
&&
\phantom{aaaaaaaaaaaaaa} 
\sum_{j=1}^ka_j = \sum_{j=1}^n M_j-lp\,,
\\
\label{> id}
&&
\phantom{aaaa}
z_{1}^{M_{1}}z_{2}^{M_{2}}\dots z_{i(l)-1}^{M_{{i(l)-1}}}
z_{i(l)}^{\sum_{j={i(l)}}^n M_{j}\,-\,lp} \ >_\id  \ z_1^{a_1}\dots z_{k-1}^{a_{k-1}}z_k^{a_k}\,.
\eean
The inequality $k\geq i(l)$ is impossible due to \eqref{cle di}, \eqref{i(l)}.
The inequality $k < i(l)$ is impossible due to \eqref{cle di}, \eqref{> id}.
\end{proof}

\subsection{$L$-admissible solutions and filtration on  space of all polynomial solutions}

Let $L=$ $(L_1$,\dots, $L_n)\in \Z^n_{\geq0}$. Let $I(z)$ be a polynomial in $z$ with coefficients in $\F_p^n$.
We say that $I(z)$ is {\it $L$-admissible} if
\bean
\label{L adm}
\frac{\der^{L_j+1}I}{\der z_i^{L_j+1}}(z) = 0,\qquad j=1,\dots,n\,.
\eean
Denote by $\mc M_L$ the $\F_p[z^p]$-module of all $L$-admissible polynomial solutions of \eqref{KZ}. 

\vsk.2>
For example,   $\mc M_{(0,\dots,0)}$  consists of polynomial solutions 
$I(z)=(I_1(z)$, \dots, $I_n(z))$ lying in $(\F_p[z^p])^n$, in other words,
it consists
of all $I(z)$, such that 
\bea
 \sum_{j \ne i}   \frac{\Omega_{ij}}{z_i - z_j}  I(z) =0\,,
\quad i = 1, \dots , n,
\qquad
m_1I_1(z)+\dots+m_nI_{n}(z)=0.
\eea
In particular, if system \eqref{KZ} has constant solutions then they lie in  
 $\mc M_{(0,\dots,0)}$ .

\vsk.2>
The modules $\mc M_L$ form a filtration on the module
 of all polynomial solutions of system \eqref{KZ}. Namely,
if $L=(L_1,\dots, L_n)$, $L'=(L_1',\dots, L_n')$ and $L_j\leq L_j'$ for all $j$, then
$\mc M_L \subset \mc M_{L'}$.

\begin{thm}
\label{thm L-adm}
Let $L=(M_1+1,\dots, M_n+1)$. Then $\mc M_L$ coincides with the module $\mc M$ of  $\F_p$-hypergeometric solutions 
of system \eqref{KZ}.
\end{thm}

\begin{proof}
Any $\F_p$-hypergeometric solution $I^{[l]}(z)$ lies in $\mc M_L$ by construction. We show that any element  $I(z)\in \mc M_L$
is a linear combination of the $\F_p$-hypergeometric
 solutions with coefficients in $\F_p[z^p]$. Since system \eqref{KZ} is homogeneous, 
 it is enough to assume that $I(z)$ is a homogeneous polynomial.

Assume that $I(z)$ is a homogeneous polynomial.
 Let $(0,\dots,0, C_i,\dots,C_n) z_1^{d_1}\dots z_n^{d_n}$ be the $\id$-leading term of $I(z)$,
 where $1\leq i < n$, $C_i\ne 0$.
Divide each $d_j$ by $p$ with remainder,
\bea
d_j = q_j p+ r_j\,,\qquad 0\leq r_j <p, \qquad j=1,\dots,n\,.
\eea
Then $r_j \leq  M_j$, $j=1,\dots,n$, since $I(z)\in \mc M_L$,
and $(r_1,\dots, r_n)$ has the form 
\\
$(M_1,\dots, M_{i-1}, r_i, 0,\dots,0)$,\, $r_i\ne M_i$,
 by Theorem \ref{lem C0}.  
We have
\bea
 r_i\, =\, \sum_{j=i}^n M_j\, -\, lp\,
\eea
for some positive integer $l$, by Corollary \ref{cor deg le}.

Consider the $\F_p$-hypergeometric solution $I^{[l]}(z)$. By Theorem \ref{lem C0} the 
$\id$-leading term of $I^{[l]}(z)$
is
\bea
(0, \dots, 0, C^l_i,\dots,C_n^l) \,z_1^{M_1}\dots z_{i-1}^{M_{i-1}}z_i^{r_i}\,,
\eea
where 
\bean
\label{CcC}
(0,\dots,0, C_i,\dots,C_n) \,=\, c\, (0, \dots, 0, C^l_i,\dots,C_n^l)\,.
\eean
for some $c \in \F_p^\times$.  Both 
$I(z)$ and  $c z_1^{q_1p}\dots z_n^{q_np} I^{[l]}(z)$ belong to $\mc M_L$ and have the same leading term.
Hence the leading monomial of the difference
$I(z) - c z_1^{q_1p}\dots z_n^{q_np} I^{[l]}(z)$ 
is lexicographically smaller than the leading monomial $z_1^{d_1}\dots z_n^{d_n}$  of $I(z)$.
Notice that the difference is also a homogeneous polynomial.

 Iterating this procedure, which decreases the leading monomial, we present
$I(z)$ as a linear combination of the $\F_p$-hypergeometric
 solutions with coefficients in $\F_p[z^p]$. 
\end{proof}

\subsection{Ample reduction} 

\begin{thm}
\label{thm am co}

Assume that system \eqref{KZ} has ample reduction for a prime $p$. 
Then any polynomial solution $I(z)$ of system \eqref{KZ} belongs to 
the module of $\F_p$-hypergeometric solutions.

\end{thm}

\begin{proof} System \eqref{KZ} is homogeneous. Hence  it is enough to prove 
the theorem assuming that $I(z)$ is a homogeneous polynomial solution.

Let $I_\id(z)=C z_1^{d_1}\dots z_n^{d_n}$ be the $\id$-leading term of $I(z)$. 
By Corollary \ref{cor le co} the $\id$-leading term has the form described in Theorem
\ref{lem C0}.  The ampleness of the reduction implies there
 exists an $\F_p$-hypergeometric solution
$I^{[l]}(z),$ whose $\id$-leading $I^{[l]}_\id(z)$ has the property
\bea
I_\id(z)=c z_1^{a_1}\dots z_n^{a_n} I^{[l]}_\id(z)\,,
\eea
where $c\in \F_p$ and $a_1,\dots,a_n\in p\Z_{\geq 0}$.  The difference 
$I(z)\,- \,c z_1^{a_1}\dots z_n^{a_n} I^{[l]}(z)$ is a homogeneous polynomial solution of system \eqref{KZ}
with $\id$-leading term  $>_\id$-smaller  than the $\id$-leading term of $I(z)$.
Iteration of this procedure  implies the theorem.
\end{proof}

\end{document}